\documentclass[11pt]{amsart}
\usepackage{mathrsfs}
\usepackage{amsfonts}
\usepackage{latexsym,amsmath,amssymb}

 \renewcommand{\epsilon}{\varepsilon}

 \newtheorem{theorem}{Theorem}[section]
 
 \newtheorem{lemma}[theorem]{Lemma}
 
 \newtheorem{Corollary}[theorem]{Corollary}
 
 \newtheorem{proposition}[theorem]{Proposition}

 \newtheorem{deff}[theorem]{Definition}
 
 \newcommand{\bth}{\begin{theorem}}
 \newcommand{\ble}{\begin{lemma}}
 \newcommand{\bcor}{\begin{corr}}
 \newcommand{\bdeff}{\begin{deff}}
 \newcommand{\bprop}{\begin{proposition}}
 \newcommand{\ele}{\end{lemma}}
 \newcommand{\ecor}{\end{corr}}
 \newcommand{\edeff}{\end{deff}}
 
 \newcommand{\eprop}{\end{proposition}}

 \renewcommand{\Pi}{\varPi}

 \renewcommand{\epsilon}{\varepsilon}

\numberwithin{equation}{section}

\newtheorem{lem}{Lemma}[section]

\newtheorem{definition}[lem]{Definition}

\title{Lagrangian Mean Curvature Flow In Pseudo-Euclidean Space}
\author{R.L. Huang}
\address{School of Mathematics, Fudan University, Shanghai 200433, People's Republic of China, E-mail: huangronglijane@yahoo.cn}
\date{}

\begin{document}
\maketitle

\begin{abstract}
We establish the longtime existence and convergence results of
the mean curvature flow of entire Lagrangian graphs in Pseudo-Euclidean space
which is related to Logarithmic gradient flow.
\end{abstract}



\section{Introduction}

The mean curvature flow in high codimension has been studied
extensively in the last few years (cf. \cite{AN}, \cite{CL1}, \cite{CL2}, \cite{KS}, \cite{KW}, \cite{X1},\cite{X3}).  In this
paper we consider the Lagrangian mean curvature flow in
Pseudo-Euclidean space .

Let $\mathbb{R}^{2n}_{n}$ be an 2$n$-dimensional Pseudo-Euclidean
space with the index $n$. The indefinite flat metric on
$\mathbb{R}^{2n}_{n}$ (cf. \cite{X2}) is defined by
\begin{equation*}
ds^{2}=\sum_{i=1}^{n}(dx^{i})^{2}-\sum_{\alpha=n+1}^{2n}(dx^{\alpha})^{2}
\end{equation*}

 Consider the logarithmic gradient flow (cf. \cite{KT}) on
 $\mathbb{R}^{n}$:
\begin{equation}\label{e1.1}
\left\{ \begin{aligned}\frac{\partial u}{\partial t}-\frac{1}{n}\ln
\mathbf{det}D^{2}u&=0,
& t>0,\quad x\in \mathbb{R}^{n}, \\
 u&=u_{0}(x), & t=0,\quad x\in \mathbb{R}^{n}.
\end{aligned} \right.
\end{equation}
By Proposition \ref{p2.1} there exist a family of diffeomorphisms
$$ r_{t}: \mathbb{R}^{n}\rightarrow \mathbb{R}^{n}$$
such that
$$F(x,t)=(r_{t},  Du(r_{t},t)) \subset
\mathbb{R}^{2n}_{n}$$ is a solution to the mean curvature flow of a
complete space-like submanifold in pseudo-Euclidean space
\begin{equation}\label{e1.2}
\left\{ \begin{aligned}\frac{dF}{dt}&=\overrightarrow{H}, \\
F(x,0)&=F_{0}(x),
\end{aligned} \right.
\end{equation}
where $\overrightarrow{H}$ is the mean curvature vector of the
submanifold $F(x,t)\subset \mathbb{R}^{2n}_{n}$ at $F(x,t)$ with
$$F_{0}(x)=(x, Du_{0}(x)).$$

\begin{deff} Assume  function $u_{0}(x)\in C^{2}(\mathbb{R}^{n})$. We call
$u_{0}(x)$ satisfying

\noindent (i) {\bf Condition A} if
\begin{equation*}
u_{0}(x)=\frac{u_{0}(Rx)}{R^{2}},\,\,\, \forall R>0.
\end{equation*}

\noindent (ii) {\bf Condition B} if
\begin{equation*}
\Lambda I\geq D^{2}u_{0}(x)\geq \lambda I,\qquad x\in
\mathbb{R}^{n},
\end{equation*}
where  $ \Lambda, \lambda$ are positive constants
 and $I$ is the identity matrix.
\end{deff}

We now state the main theorems  of this paper.
\begin{theorem}\label{t1.1}
Let $u_{0}:\mathbb{R}^{n}\rightarrow \mathbb{R}$ be a $C^{2}$
function which satisfies  Condition B.
Then  there exists a unique strictly
convex solution  of (\ref{e1.1}) such that
\begin{equation}\label{e1.3}
u(x,t)\in C^{\infty}(\mathbb{R}^{n}\times (0,+\infty))\cap
C(\mathbb{R}^{n}\times [0,+\infty))
\end{equation}
where $u(\cdot,t)$ satisfies Condition B.  More generally, for  $l=\{3,4,5\cdots\}$ and $\epsilon_{0}>0$,  there holds
\begin{equation}\label{e1.4}
 \parallel D^{l}u(\cdot,t)\parallel^{2}_{C(\mathbb{R}^{n})}\leq
 C, \qquad \forall t\in (\epsilon_{0},+\infty),
\end{equation}
where $C$  depends only on $n, \lambda, \Lambda, \dfrac{1}{\epsilon_{0}}.$
\end{theorem}

Consider the following  Monge-Amp\`{e}re type equation
\begin{equation}\label{e1.5}
\det D^{2}u=\exp\{n(u-\frac{1}{2}\sum_{i=1}^{n}x_{i}\frac{\partial u}{\partial x_{i}})\}.
\end{equation}
According to the definition in \cite{CM}, we can show that an entire
solution to  (\ref{e1.5}) is  a
self-expending solutions to Lagrangian mean curvature flow in
Pseudo-Euclidean space.
\begin{theorem}\label{t1.2}
There exists a one-to-one correspondence between smooth
self-expending solutions satisfying condition B and functions
satisfying  Condition A and Condition B.
\end{theorem}

The following theorem shows that  we can obtain the self-expending solution by the Logarithmic gradient flow.
\begin{theorem}\label{t1.3}
Let $u_{0}:\mathbb{R}^{n}\rightarrow \mathbb{R}$ be a $C^{2}$
function which satisfies Condition B. Assume that
\begin{equation}\label{e1.6}
\lim_{\lambda\rightarrow+\infty}\tau^{-2}u_{0}(\tau x)=U_{0}(x)
\end{equation}
for some $U_{0}(x)\in C^{2}(\mathbb{R}^{n})$. Let $u(x,t)$ and
$U(x,t)$ be solutions to (\ref{e1.1}) with initial data $u_{0}(x)$
and $U_{0}(x)$ respectively. Then,
\begin{equation}\label{e1.7}
\lim_{t\rightarrow+\infty}t^{-1}u(\sqrt{t}x, t)=U(x,1).
\end{equation}
Here the convergence is uniform and smooth  in any compact subset of
$\mathbb{R}^{n}$, and $U(x,1)$ is a self-expanding solution of
(\ref{e1.2}).
\end{theorem}
To describe the asymptotic behavior of the Lagrangian mean curvature flow (\ref{e1.2}), we will prove

\begin{theorem}\label{t1.4}
Suppose that  $u_{0}:\mathbb{R}^{n}\rightarrow \mathbb{R}$ be a
$C^{2}$ function which satisfies Condition B and $\sup_{x\in \mathbb{R}^{n}}|Du_{0}(x)|^{2}<+\infty$. Then the evolution equations
of mean curvature flow (\ref{e1.2}) has a longtime smooth solution and the
graph $(x, Du(x,t))$ converges to a plane in $\mathbb{R}^{2n}_{n}$ as $t$ goes
to infinity. If we assume in addition that $|Du_{0}(x)|\rightarrow
0$ as $|x|\rightarrow \infty $, then the graph $(x, D u(x,t))$
converges smoothly on compact sets to the coordinate plane $(x,0)$
in $\mathbb{R}^{2n}_{n}$.
\end{theorem}

\section{Logarithmic gradient  flow related to Lagrangian mean curvature flow }

Let $(x^{1},\cdots, x^{n}; y^{1},\cdots, y^{n})$ be null coordinates
in  $\mathbb{R}^{2n}_{n}$.  Then the indefinite  metric (cf. \cite{X2})
is defined by
\begin{equation}\label{e2.1a}
ds^{2}=\frac{1}{2}\sum_{i=1}^{n}dx^{i}dy^{i}
\end{equation}
Suppose $u$ be a smooth convex function. We consider the graph M of
$\nabla u$, defined by
\begin{equation*}
(x^{1},\cdots, x^{n};\frac{\partial u}{\partial x^{1}},\cdots,
\frac{\partial u}{\partial x^{n}}).
\end{equation*}
The induce Rimannian metric on M is defined by
\begin{equation*}
ds^{2}=\frac{\partial^{2}u}{\partial x^{i}\partial
x^{j}}dx^{i}dx^{j}.
\end{equation*}
Choose a tangent frame field $\{e_{1},\cdots,e_{n}\}$ along M, where
\begin{equation*}
e_{i}=\frac{\partial}{\partial x^{i}}+\frac{\partial^{2}u}{\partial
x^{i}\partial x^{j}}\frac{\partial}{\partial y^{j}}.
\end{equation*}
We use $\langle\,\,,\,\rangle$ to denote the  inner product induced
from (\ref{e2.1a}).  Then
\begin{equation*}
\langle e_{i},\,e_{j}\rangle=\frac{\partial^{2}u}{\partial
x^{i}\partial x^{j}}.
\end{equation*}
Let $\{\eta_{1},\cdots,\eta_{n}\}$ be the normal frame field of M in
$\mathbb{R}^{2n}_{n}$ defined by
\begin{equation*}
\eta_{i}=\frac{\partial}{\partial
x^{i}}-\frac{\partial^{2}u}{\partial x^{i}\partial
x^{j}}\frac{\partial}{\partial y^{j}}
\end{equation*}
with
\begin{equation*}
\langle \eta_{i},\,\eta_{j}\rangle=-\frac{\partial^{2}u}{\partial
x^{i}\partial x^{j}}.
\end{equation*}
The mean curvature vector of M is given by
\begin{equation*}
\overrightarrow{H}=-\frac{1}{2ng}\frac{\partial g}{\partial
x^{l}}g^{lk}\eta_{k},
\end{equation*}
where $g=\mathrm{det}D^{2}u$.

If  $u(x,t)\in C^{3,\frac{3}{2}}$, $u$ is strictly convex function
in $\mathbb{R}^{n}$  and
\begin{equation*}
F(x(t),t)=(x^{1},\cdots, x^{n};\frac{\partial u}{\partial
x^{1}},\cdots, \frac{\partial u}{\partial x^{n}})\end{equation*}
satisfies (\ref{e1.2}). Then
\begin{equation*}
\frac{dx^{i}}{dt}=-\frac{1}{2ng}\frac{\partial g}{\partial
x^{l}}g^{li},\qquad \frac{du_{j}}{dt}=\frac{1}{2ng}\frac{\partial
g}{\partial x^{l}}g^{lk}\frac{\partial^{2}u}{\partial x^{k}\partial
x^{j}},\qquad i,j=1,2, \cdots,n.
\end{equation*}
where $\displaystyle u_{j}=\frac{\partial u}{\partial x^{j}},
\,\,[g_{ij}]=D^{2}u,\,\, [g^{ij}]=[g_{ij}]^{-1}$. However,
\begin{equation*}
\frac{du_{j}}{dt}=\frac{\partial u_{j}}{\partial t}+\frac{\partial
u_{j}}{\partial x^{k}}\frac{dx^{k}}{dt},\qquad j=1,2, \cdots,n.
\end{equation*}
So that
\begin{equation*}\aligned
\frac{\partial u_{j}}{\partial t}&=\frac{1}{2ng}\frac{\partial
g}{\partial x^{l}}g^{lk}\frac{\partial^{2}u}{\partial x^{k}\partial
x^{j}}+\frac{1}{2ng}\frac{\partial g}{\partial
x^{l}}g^{lk}\frac{\partial^{2}u}{\partial x^{k}\partial x^{j}}\\
&=\frac{1}{ng}\frac{\partial g}{\partial x^{l}}g^{lk}g_{kj}\\
&=\frac{1}{n}\frac{\partial}{\partial x^{j}}\ln g, \qquad j=1,2,
\cdots,n.
\endaligned
\end{equation*}
Then $u(x,t)$ satisfies (\ref{e1.1}).

Conversely, if  $u(x,t)\in C^{2,1}$ and $u$ is strictly convex
function in $\mathbb{R}^{n}$. Then we define in the obvious way
\begin{equation*}
\tilde{F}(x,t)=(x^{1},\cdots, x^{n};\frac{\partial u}{\partial
x^{1}},\cdots, \frac{\partial u}{\partial x^{n}}).
\end{equation*}
Let $ r: \mathbb{R}^{n}\times [0,T)\rightarrow \mathbb{R}^{n}$ be
the solution of the following system of ordinary differential
equations:
\begin{equation*}\left\{ \begin{aligned}
\frac{dx^{i}}{dt}&=-\frac{1}{2ng}\frac{\partial g}{\partial
x^{l}}g^{li},\qquad & i=1,2, \cdots,n,\\
 x^{i}(0)&=x^{i},\qquad & i=1,2, \cdots,n.
\end{aligned} \right.
\end{equation*}
Then  $ r_{t}$ be a family of diffeomorphisms  $
\mathbb{R}^{n}\rightarrow \mathbb{R}^{n}$ and
$F(x,t)=\tilde{F}(r(x,t),t)$ be the solution of (\ref{e1.2}).

In summary by the regularity theory of parabolic equation we have the following result:

\begin{proposition}\label{p2.1}
Let $u_{0}: \mathbb{R}^{n}\rightarrow \mathbb{R}$ be a strictly
convex $C^{2}$ function. Then (\ref{e1.1}) has a  strictly convex smooth
solution on $\mathbb{R}^{n}\times(0,T)$ with initial condition
$u(x,0)=u_{0}(x)$ if and only if (\ref{e1.2}) has a smooth solution
$F(x,t)$ on $\mathbb{R}^{n}\times(0,T)$ with strictly convex
potential and with initial condition $F(x,0)=(x,\nabla u_{0}(x))$.
In particular, there exists a smooth family of diffeomorphisms
$r(x,t) : \mathbb{R}^{n}\rightarrow \mathbb{R}^{n}$ for $t\in [0,T)$
such that $F(x,t)=(r(x,t), \nabla u(r(x,t),t))$ solves (\ref{e1.2}) on
$\mathbb{R}^{n}\times[0,T)$.
\end{proposition}

A solution $F(\cdot,t)$ of (\ref{e1.2}) is called self-expending if
it has the form
\begin{equation}\label{e2.1}
M_{t}=\sqrt{t}M_{1}\quad \mathrm{for}\,\, \mathrm{all}\,\, t>0,
\end{equation}
where $M_{t}=F(\cdot,t)$.

Assume that $F(x,t)$ is a self-expending  solution of (\ref{e1.2}).
Following Proposition 2.1 ,  $u(x,t)$ satisfies
\begin{equation}\label{e2.3}
\frac{\partial u}{\partial t}-\frac{1}{n}\ln
\det D^{2}u=0,
 \,\,\,t>0,\,\,\,\quad x\in \mathbb{R}^{n}.
\end{equation}
Hence,
\begin{equation*}
D(u(x,t)-tu(\frac{x}{\sqrt{t}}, 1))=0,
\end{equation*}
$\mathrm{i.e}$
\begin{equation}\label{e2.4}
u(x,t)=tu(\frac{x}{\sqrt{t}}, 1),\,\,\,t>0.
\end{equation}
Thus combining (\ref{e2.3}), (\ref{e2.4}) and letting $t=1$, we can
verify that  $u(x,1)$ satisfies (\ref{e1.5}).

We want to use the continue methods to prove the solvability of
(\ref{e1.1}).
\begin{definition}\label{d1.2}
Given $T>0$. Let $\tau\in [0,1]$. We say $u\in
C^{5,\frac{5}{2}}(\mathbb{R}^{n}\times(0,T))\cap
C(\mathbb{R}^{n}\times[0,T) )$ is a solution of $(\star_{\tau})$ if
$u$ satisfies
\begin{equation}\label{e2.2}
\left\{ \begin{aligned}\frac{\partial u}{\partial
t}-\frac{\tau}{n}\ln \mathbf{det}D^{2}u-(1-\tau)\triangle u&=0,
 &t>0,\,\,\, x\in \mathbb{R}^{n}, \\
\,\,\,\,\, u&=u_{0}(x), &t=0,\,\,\, x\in \mathbb{R}^{n}.
\end{aligned} \right.
\end{equation}
\end{definition}

Set
\begin{equation*}
u_{0}(x,t)=\frac{1}{(4\pi
t)^{\frac{n}{2}}}\int_{\mathbb{R}^{n}}u_{0}(y)\exp[-\frac{|x-y|^{2}}{4t}]dy.
\end{equation*}
Clearly $u_{0}(x,t)$ is a solution of (\ref{e2.2}) with $\tau=0$. Let
\begin{equation*}
\mathrm{I}=\{\tau\in [0,1]: (\star_{\tau}) \mbox{\ has a solution \ }\}.
\end{equation*}
The long time existence of  the flow (\ref{e1.2}) holds if we can show that $\mathrm{I}$ is both closed and open.
To prove that the classical solution of  (\ref{e1.1}) must be strictly convex we need  the following conclusion
which is proved by Pierre-Louis Lions, Marek Musiela (cf. Theorem 3.1 in \cite{PM}).
\begin{lemma}\label{p2.3}
Let $u:  \mathbb{R}^{n}\times [0,T)\rightarrow \mathbb{R}$ be a
solution of a fully nonlinear equations of the form
\begin{equation*}
\frac{\partial u}{\partial t}=F(D^{2}u)
\end{equation*}
where $F$ is a $C^{2}$  function defined on the cone $\Gamma_{+}$ of
definite symmetric  matrices, which is monotone increasing (
that is, $F(A)\leq F(A+B)$ whenever $B$ is a positive definite
matrix), and such that the function
\begin{equation*}
F^{*}(A)=-F(A^{-1})
\end{equation*}
is concave on $\Gamma_{+}$. If $D^{2}u\geq 0$
  everywhere on  $\mathbb{R}^{n}$
 for $t=0$. Then   $D^{2}u\geq 0$  everywhere on  $\mathbb{R}^{n}$
 for $0\leq t\leq T$.
\end{lemma}
From Lemma \ref{p2.3}, we obtain

 \begin{Corollary}\label{c2.0} Suppose that $u:  \mathbb{R}^{n}\times [0,T)\rightarrow \mathbb{R}$ be a
 solution of a fully nonlinear equations of the form
 \begin{equation*}
 \frac{\partial u}{\partial t}=F(D^{2}u)
 \end{equation*}
 where $F$ satisfies the conditions in Lemma \ref{p2.3} and $F$ is concave on the cone $\Gamma_{+}$.
  If $\lambda I\leq D^{2}u\leq \Lambda I$
  $($for some $0<\lambda<\Lambda$$)$ everywhere on  $\mathbb{R}^{n}$
  for $t=0$. Then   $\lambda I\leq D^{2}u\leq \Lambda I$
   everywhere on  $\mathbb{R}^{n}$
  for $0\leq t\leq T$.
 \end{Corollary}
 \begin{proof}

 Step 1.  We will prove that $ D^{2}u\geq\lambda I$  everywhere on  $\mathbb{R}^{n}$ for $0\leq t\leq T$.

 Set $\bar{u}=u-\frac{\lambda}{2}|x|^{2}$. Then $\bar{u}$ satisfies
 \begin{equation*}
 \frac{\partial \bar{u}}{\partial t}=F(D^{2}\bar{u}+\lambda I)
 \end{equation*}
 with $D^{2}\bar{u}|_{t=0}\geq 0$. Define
 $$\bar{F}(D^{2}\bar{u})=F(D^{2}\bar{u}+\lambda I),$$
 $$\bar{F}^{*}(A)=-F(A^{-1}+\lambda I),$$
 $$\bar{F}^{*}(\lambda_{1}, \lambda_{2}, \cdots, \lambda_{n})=-F(\lambda^{-1}_{1}+\lambda, \lambda^{-1}_{2}+\lambda, \cdots, \lambda^{-1}_{n}+\lambda),$$
 $$\Sigma=\{\lambda_{1}>0, \lambda_{1}>0,\cdots, \lambda_{1}>0\}.$$
 It follows from  \cite{LLJ2} that $\bar{F}^{*}(A)$ is concave on $\Gamma_{+}$ if and only if $\bar{F}^{*}(\lambda_{1}, \lambda_{2}, \cdots, \lambda_{n})$ is concave on $\Sigma$. Note that for all $\xi\in \mathbb{R}^{n}$,
 $$\frac{\partial^{2}\bar{F}^{*}}{\partial\lambda_{i}\partial\lambda_{i}}\xi_{i}\xi_{j}=-\bar{F}_{ij}
 \bar{\xi}_{i}\bar{\xi}_{j}-2\bar{F}_{i}\lambda_{i}\bar{\xi}_{i}^{2}$$
 where $\bar{\xi}_{i}=\dfrac{\xi_{i}}{\lambda^{2}_{i}}.$
 Since $F^{*}(A)=-F(A^{-1})$ is concave on $\Gamma_{+}$, we have
 $$-\bar{F}_{ij}
 \bar{\xi}_{i}\bar{\xi}_{j}|_{\lambda=0}-2\bar{F}_{i}\lambda_{i}\bar{\xi}_{i}^{2}|_{\lambda=0}\leq 0.$$
 So that
 $$-\bar{F}_{ij}\bar{\xi}_{i}\bar{\xi}_{j}\leq 2\bar{F}_{i}\frac{\lambda_{i}}{1+\lambda\lambda_{i}}\bar{\xi}_{i}^{2},$$
 Clearly,
 $$\frac{\partial^{2}\bar{F}^{*}}{\partial\lambda_{i}\partial\lambda_{i}}\xi_{i}\xi_{j}=-\bar{F}_{ij}
 \bar{\xi}_{i}\bar{\xi}_{j}-2\bar{F}_{i}\lambda_{i}\bar{\xi}_{i}^{2}\leq 2\bar{F}_{i}\frac{\lambda_{i}}{1+\lambda\lambda_{i}}\bar{\xi}_{i}^{2}
 -2\bar{F}_{i}\lambda_{i}\bar{\xi}_{i}^{2}\leq 0.$$
 Such that one can apply Lemma \ref{p2.3} that $D^{2}\bar{u}\geq 0$ for for $0\leq t\leq T$.

 Step 2.  We will prove that $ D^{2}u\leq\Lambda I$  everywhere on  $\mathbb{R}^{n}$ for $0\leq t\leq T$.

 Introduce the Legendre transformation of $u$
 \begin{equation*}
 \tau=t,\,\,\,y^{i}=\frac{\partial u}{\partial x^{i}}
 ,\,\,i=1,2,\cdots,n,\,\,\,u^{*}(y^{1},\cdots,y^{n}):=\sum_{i=1}^{n}x^{i}\frac{\partial u}{\partial x^{i}}-u(x).
 \end{equation*}
 In terms of $\tau, y^{1},\cdots,y^{n}, u^{*}(y^{1},\cdots,y^{n},\tau)$, one can easily check that
 $$\frac{\partial u^{*}}{\partial\tau}=-\frac{\partial u}{\partial t},\,\,\,\,\,\,\,\,\,\,
 \frac{\partial^{2} u^{*}}{\partial y^{i}\partial y^{j}}=[\frac{\partial^{2} u}{\partial x^{i}\partial x^{j}}]^{-1}.$$
 Then $u^{*}$ is a solution of the form
 \begin{equation*}
 \frac{\partial u^{*}}{\partial \tau}=F^{*}(D^{2}u^{*}).
 \end{equation*}
 Since $F^{**}=F$ is concave on the cone $\Gamma_{+}$, using the conclusions of step 1 we obtain
 $ D^{2}u^{*}\geq\frac{1}{\Lambda} I$ for $0\leq t\leq T$ and this yields our desired results.
 \end{proof}
 For the problem (\ref{e2.2}) we have
 \begin{lemma}\label{l2.4}
 $\mathrm{I}$ is closed.
 \end{lemma}
 \begin{proof}

 Suppose that $u$ is a solution of $(\star_{\tau})$. For $A\in \Gamma_{+}$,  set
 \begin{equation*}
 F(A)=\frac{\tau}{n}\ln \mathbf{det}A+(1-\tau)\mathrm{Tr}A.
 \end{equation*}
 Let $\lambda_{1}, \lambda_{2}, \cdots, \lambda_{n}$  be the eigenvalues of $A$.  Define
 \begin{equation*}
 f(\lambda_{1}, \lambda_{2}, \cdots, \lambda_{n})=F(A)=\frac{\tau}{n}\ln \lambda_{1}\lambda_{2}\cdots\lambda_{n}
 +(1-\tau)(\lambda_{1}+\lambda_{2}+\cdots+\lambda_{n})
 \end{equation*}
 and
 \begin{equation*}
 f^{*}(\lambda_{1}, \lambda_{2}, \cdots, \lambda_{n})=F^{*}(A)=\frac{\tau}{n}\ln \lambda_{1}\lambda_{2}\cdots\lambda_{n}
 -(1-\tau)(\frac{1}{\lambda_{1}}+\frac{1}{\lambda_{2}}+\cdots+\frac{1}{\lambda_{n}}).
 \end{equation*}
 One can verify that $D^{2}f, D^{2}f^{*}$ are  negative in a cone $\Sigma=\{\lambda_{1}>0, \lambda_{1}>0,\cdots, \lambda_{1}>0\}$.
 By \cite{LLJ2}, we deduce that $F, F^{*}$ are  smooth concave functions defined on the cone $\Gamma_{+}$ of
 definite symmetric matrix matrices, which is monotone increasing.

  It follows from Corollary \ref{c2.0} that if $u_{0}(x)$ satisfies Condition B then $u(x,t)$ does so.
  For $s>0, \Omega\subset \mathbb{R}^{n}$ define
 \begin{equation*}
 \Omega_{T}=\Omega\times[0,T),\qquad \Omega_{T,s}=\Omega\times[s,T).
 \end{equation*}
 Furthermore by the regularity theory of parabolic equation (cf. \cite{G}) we have
 \begin{equation}\label{e2.3}
 \|u\|_{C^{2,1}(\bar{\Omega}_{T})}\leq C_{1},\qquad
 \|u\|_{C^{2+\alpha,\frac{2+\alpha}{2}}(\bar{\Omega}_{T,s})}\leq C_{2},
 \end{equation}
 where $0<\alpha<1$, $C_{1}$ is a positive constant depending only on $u_{0}, \Omega, T$,
 and $C_{2}$  relies on  $\lambda, \Lambda, \Omega, T, \displaystyle\frac{1}{s}$.
  By (\ref{e2.3}),   a diagonal sequence argument and
 the regularity theory of parabolic equation shows that $\mathrm{I}$ is closed.
 \end{proof}

 To prove that $\mathrm{I}$ is open we need the following lemma (cf. Theorem 17.6 in \cite{GT}).
 \begin{lemma}\label{l2.5}
 Let $\mathcal{B}_{1}, \mathcal{B}_{2}$ and $\mathbf{X}$ be Banach spaces
 and $G$ is a mapping from an open subset of $\mathcal{B}_{1}\times \mathbf{X}$
 into $\mathcal{B}_{2}$. Let $(u_{0}, \tau_{0})$ be a point in $\mathcal{B}_{1}\times \mathbf{X}$  satisfying:

 $(i)$\, $G[u_{0}, \tau_{0}]=0$;

 $(ii)$\, $G$ is continuously differentiable at $(u_{0}, \tau_{0})$;

 $(iii)$\, the partial Fr$\acute{e}$chet derivative $L=G^{1}_{(u_{0}, \tau_{0})}$ is invertible.

 Then there exists a neighbourhood $\mathcal{N}$ of $\tau_{0}$ in $\mathbf{X}$ such that the equation
 $G[u, \tau]=0$, is solvable for each $\tau\in \mathcal{N}$, with solution $u=u_{\tau}\in \mathcal{B}_{1}$.
 \end{lemma}
 Based on the implicit function theorem we have the following conclusions.
 \begin{lemma}\label{l2.6}
 $\mathrm{I}$ is open.
 \end{lemma}
 \begin{proof}

 Define the Banach spaces
 \begin{equation*}
 \mathcal{B}_{1}=C^{5,\frac{5}{2}}(\mathbb{R}^{n}\times(0,T))\cap
 C(\mathbb{R}^{n}\times[0,T) ),\qquad  \mathbf{X}=\mathbb{R},
 \end{equation*}
 \begin{equation*}
 \mathcal{B}_{2}=C^{3,\frac{3}{2}}(\mathbb{R}^{n}\times(0,T))\times
 C(\mathbb{R}^{n} ),
 \end{equation*}
 and a continuously differentiable map from  $\mathcal{B}_{1}\times \mathbf{X}$  into $\mathcal{B}_{2}$,
 \begin{equation*}
 G:\, (u, \tau)\rightarrow[\frac{\partial u}{\partial t}-\frac{\tau}{n}\ln \mathbf{det}D^{2}u-(1-\tau)\triangle u,u-u_{0}].
 \end{equation*}
 Take an open set of $\mathcal{B}_{1}\times \mathbf{X}$:
 \begin{equation*}
 \Theta= \{u| \frac{\lambda}{2} I< D^{2}u(x,t)<\frac{3\Lambda}{2} I,
 \quad u\in C^{5,\frac{5}{2}}(\mathbb{R}^{n}\times(0,T))\cap C(\mathbb{R}^{n}\times[0,T))\,\}\times (0,1).
 \end{equation*}

 Suppose that $(u_{0}, \tau_{0})\in \Theta$. Then the partial Fr$\acute{e}$chet
 derivative $L=G^{1}_{(u_{0}, \tau_{0})}$ is invertible if and only if the following cauchy problem is solvable
 \begin{equation*}\label{1.03}
 \left\{ \begin{aligned}\frac{\partial w}{\partial
 t}-\frac{\tau_{0}}{n}u^{ij}_{0}\frac{\partial^{2}w}{\partial x^{i}\partial x^{j}}-(1-\tau_{0})\triangle w&=f,
  &t>0,\,\,\, x\in \mathbb{R}^{n}, \\
 \,\,\,\,\, w&=g, &t=0,\,\,\, x\in \mathbb{R}^{n},
 \end{aligned} \right.
 \end{equation*}
 where $(f,g)\in \mathcal{B}_{2}$. Using the linear parabolic equations theory (cf. \cite{G}) we can do it.

 Thereby applying Lemma \ref{l2.5} we prove that  $\mathrm{I}$ is open.
 \end{proof}
Given $x_{0}\in \mathbb{R}^{n}$, $\kappa>0$, define
\begin{equation*}
Q_{1,x_{0}}=\{x| |x-x_{0}|\leq 1\}\times [\kappa,\kappa+1),\qquad Q_{\frac{1}{2},x_{0}}=\{x| |x-x_{0}|\leq \frac{1}{2}\}\times [\kappa+\frac{1}{4},\kappa+\frac{1}{2}),
\end{equation*}
\begin{equation*}
Q_{\frac{1}{3},x_{0}}=\{x| |x-x_{0}|\leq \frac{1}{3}\}\times [\kappa+\frac{1}{3},\kappa+\frac{5}{12}),\qquad B_{1,x_{0}}=\{|x-x_{0}|\leq1\}.
\end{equation*}

The following two lemmas which will be mentioned below may be
used repeatedly (cf. \cite{G}).

\begin{lemma}\label{p3.4}  $\mathbf{(Theorem \,\,14.7\,\, in\,\, [13])}$. Let $u:  \mathbb{R}^{n}\times [0,T)\rightarrow \mathbb{R}$ be a
classical solution of a fully nonlinear equation of the form
\begin{equation*}\label{1.03}
\left\{ \begin{aligned}\frac{\partial u}{\partial
t}-F(D^{2}u)&=0,
 &t>0,\,\,\, x\in \mathbb{R}^{n}, \\
\,\,\,\,\, u&=u_{0}(x), &t=0,\,\,\, x\in \mathbb{R}^{n},
\end{aligned} \right.
\end{equation*}
where $F$ is a $C^{2}$ concave function defined on the cone $\Gamma_{+}$ of
definite symmetric matrix matrices, which is monotone increasing with
\begin{equation*}
\lambda I \leq\frac{\partial F}{\partial r_{ij}}\leq \Lambda I.
\end{equation*}
Then there existence $0<\alpha<1$ such that
\begin{equation*}
[D^{2}u]_{C^{\alpha,\frac{\alpha}{2}}(\bar{Q}_{\frac{1}{2},x_{0}})}\leq C|D^{2}u|_{C^{0}(\bar{Q}_{1,x_{0}})},
\end{equation*}
where $\alpha, C$ are positive constants depending only on $n, \lambda, \Lambda, \frac{1}{\kappa}.$
\end{lemma}

\begin{lemma}\label{p3.5} $\mathbf{(Theorem \,\,4.9\,\, in\,\, [13])}$.
Let $v:  \mathbb{R}^{n}\times [0,T)\rightarrow \mathbb{R}$ be a
classical solution of a linear parabolic equation of the form
\begin{equation*}\label{1.03}
\left\{ \begin{aligned}\frac{\partial v}{\partial
t}-a^{ij}v_{ij}&=0,
 &t>0,\,\,\, x\in \mathbb{R}^{n}, \\
\,\,\,\,\, v&=v_{0}(x), &t=0,\,\,\, x\in \mathbb{R}^{n},
\end{aligned} \right.
\end{equation*}
where there exist positive constants $C$ such that
\begin{equation*}
\lambda I \leq a^{ij}\leq \Lambda I,\qquad [a^{ij}]_{C^{\alpha}(\bar{Q}_{\frac{1}{2},x_{0}})}\leq C.
\end{equation*}
Then there holds
\begin{equation*}
|D^{2}v|_{C^{0}(\bar{Q}_{\frac{1}{3},x_{0}})}+[D^{2}v]_{C^{\alpha,\frac{\alpha}{2}}(\bar{Q}_{\frac{1}{3},x_{0}})}
\leq C_{3}|v_{0}|_{C^{0}(\bar{B}_{1,x_{0}})},
\end{equation*}
where $ C_{3}$ are positive constants depending only on $n, \lambda, \Lambda $ and $C, \frac{1}{\kappa}$.
\end{lemma}

{\bf Proof of Theorem \ref{t1.1}:}

Using  the linear parabolic equations theory (cf. \cite{G})
and combining Lemma \ref{l2.4} with Lemma \ref{l2.6}
we conclude that  there exists a unique strictly convex solution  of (\ref{e1.1})
satisfying (\ref{e1.3}) and $u(\cdot,t)$ satisfies Condition B.

By Lemma \ref{p3.4} we get
\begin{equation*}
[D^{2}u]_{C^{\alpha,\frac{\alpha}{2}}(\bar{Q}_{\frac{1}{2},x_{0}})}\leq C,
\end{equation*}
where $ C$ are positive constants depending only on $n, \lambda, \Lambda $ and $ \frac{1}{\kappa}$.

For $m\in \{1,2,\cdots,n\}$, set $v=\dfrac{\partial u}{\partial x_{m}}$. Then $v$ satisfies \begin{equation*}\label{1.03}
\left\{ \begin{aligned}\frac{\partial v}{\partial
t}-\frac{1}{n}u^{ij}v_{ij}&=0,
 &t>0,\,\,\, x\in \mathbb{R}^{n}, \\
\,\,\,\,\, v&=v_{0}(x), &t=0,\,\,\, x\in \mathbb{R}^{n}.
\end{aligned} \right.
\end{equation*}
Such that by Lemma \ref{p3.5} we have
\begin{equation}\label{e3.8}
|D^{3}u|_{C^{0}(\bar{Q}_{\frac{1}{3},x_{0}})} \leq C|Du_{0}|_{C^{0}(\bar{B}_{1,x_{0}})}.
\end{equation}
Let
\begin{equation*}
\tilde{v}(x,t)=v-\frac{\partial u_{0}}{\partial x^{m}}(x_{0}).
\end{equation*}
It is easy to see that $\tilde{v}$ satisfies
\begin{equation*}\label{1.03}
\left\{ \begin{aligned}\frac{\partial \tilde{v}}{\partial
t}-\frac{1}{n}u^{ij}\tilde{v}_{ij}&=0,
 &t>0,\,\,\, x\in \mathbb{R}^{n}, \\
\,\,\,\,\, \tilde{v}&=\tilde{v}(x,0), &t=0,\,\,\, x\in \mathbb{R}^{n}.
\end{aligned} \right.
\end{equation*}
Then by (\ref{e3.8}) and the mean value theorem  we arrive at
\begin{equation*}
|D^{3}u|_{C^{0}(\bar{Q}_{\frac{1}{3},x_{0}})}
\leq C|Du_{0}-Du_{0}(x_{0})|_{C^{0}(\bar{B}_{1,x_{0}})}\leq C.
\end{equation*}
Using the similar methods we obtain (\ref{e1.4}) for $l=\{3,4,5\cdots\}$.
\qed

{\bf Proof of Theorem \ref{t1.2}:}

The main idea comes from \cite{ACH2} and we present here for completeness.

 Case 1. If $u_{0}$ satisfies Condition A and B. Then by Theorem \ref{t1.1},
 there exists a unique smooth solution $u(x,t)$ to (\ref{e1.1})  for all $t>0$
 with initial data $u_{0}$. One can verify that
\begin{equation*}
u_{R}(x,t):=R^{-2}u(Rx,R^{2}t)
\end{equation*}
 is a solution to (\ref{e1.1})  with initial data
\begin{equation*}
u_{R}(x,0):=R^{-2}u_{0}(Rx)=u_{0}(x).
\end{equation*}
Here we have used that $u_{0}$ satisfies Condition A.  Since $u_{R}(x,0)=u_{0}$,
the uniqueness result in Theorem\ref{t1.1}  implies
\begin{equation*}
u(x,t)=u_{R}(x,t)
\end{equation*}
for any $R>0$. Therefore $u(x,t)$  satisfies (\ref{e2.4}), and hence
$u(x,1)$ solves (\ref{e1.5}). In other words, $u(x,1)$ is a smooth
self-expending solution.

Case 2. If  $v$ is a smooth solution to (\ref{e1.5})  satisfying Condition B.
Define $u(x,t)$ for $t>0$ by
\begin{equation*}
u(x,t)=tv(\frac{x}{\sqrt{t}}).
\end{equation*}
By the definition of $v$ we conclude that  $u(x,t)$ satisfies the evolution equation (\ref{e1.1}).
Now we claim that  $\lim_{t\rightarrow0}u(x,t)$ exists. To see this,
 note that $u(0,t)=tv(0)$ for $t>0$, so
\begin{equation*}
\lim_{t\rightarrow 0}u(0,t)=0.
\end{equation*}
Moreover, it is clear that
\begin{equation*}
Du(0,t)=\sqrt{t}Dv(0),
\end{equation*}
Since
\begin{equation}\label{e3.90}
D^{2}u(x,t)=D^{2}_{x}(tv(\frac{x}{\sqrt{t}}))=D^{2}v(\frac{x}{\sqrt{t}}),
\end{equation}
we get    for any $t>0$,
\begin{equation*}
 \lambda I\leq D^{2}u(x,t)\leq \Lambda I,\qquad x\in \mathbb{R}^{n}.
 \end{equation*}
Using the medium theorem twice, we have
\begin{equation*}\aligned
u(x,t)=&u(x,t)-u(0,t)+u(0,t)\\
=&<Du(\xi,t),x>+u(0,t)\\
=&<Du(\xi,t)-Du(0,t),x>+<Du(0,t),x>+u(0,t)\\
=&\sum_{i,j=1}^{n}\zeta_{i}u_{ij}x_{j}+<Du(0,t),x>+u(0,t).
\endaligned
\end{equation*}
Applying Caffarelli's regularity theory of Monge-Amp\`{e}re type
equation ( cf. \cite{L1}, \cite{L2} )  and interior Schauder estimates to the
equation (\ref{e1.5}), we may then conclude that, for any sequence
$t_{i}\rightarrow 0$, there is a subsequence $t_{k_{i}}$, such that
$u(x, t_{k_{i}})$ converges in $C^{2,\alpha}$ uniformly in compact
subsets of $\mathbb{R}^{n}$ for any $0<\alpha<1$. This limit is in
fact independent of the choice of the subsequence $\{t_{k_{i}}\}$.
Indeed, let $u_{1}$ and $u_{2}$ be two such limits for  subsequences
$\{t_{i}\}$ and $\{t'_{i}\}$ respectively. Since $u(x,t)$ is a
solution to (\ref{e1.1}) , $\dfrac{\partial u}{\partial t}$ is
uniformly bounded for any $t>0$ and $x\in \mathbb{R}^{n}$. Thus for
any $x\in \mathbb{R}^{n}$  we may have
\begin{equation*}
|u(x,t_{i})-u(x,t'_{i})|\leq C|t_{i}-t'_{i}|
\end{equation*}
for some constant $C$ independent of $i$. Letting $i\rightarrow
\infty$, we conclude that $u_{1}(x)=u_{2}(x).$ So for different
sequences $\{t_{i}\}$ converging to $0$, the limit is unique. Let
\begin{equation*}
u_{0}(x)=\lim_{t\rightarrow 0}u(x,t).
\end{equation*}
Then it follows from (\ref{e3.90}) that $u_{0}$ satisfies Condition B.
Further,
\begin{equation*}
\frac{1}{R^{2}}u_{0}(Rx)=\frac{1}{R^{2}}\lim_{t\rightarrow 0}tv(\frac{Rx}{\sqrt{t}})
=\lim_{t\rightarrow 0}(\frac{\sqrt{t}}{R})^{2}v(\frac{Rx}{\sqrt{t}})=u_{0}(x).
\end{equation*}
Therefore $u_{0}(x)$ satisfies Condition A.

From the above two cases we see that  Theorem \ref{t1.2} is established.
\qed

We present here the proof of Theorem \ref{t1.3} by the methods of \cite{ACH2}.

{\bf Proof of Theorem \ref{t1.3}:}

Assume that
\begin{equation*}
U_{0}(x)=\lim_{R\rightarrow+\infty}R^{-2}u_{0}(Rx).
\end{equation*}
So $U(x,0)$ satisfies Condition B. As in the proof of Theorem \ref{t1.2},  we obtain
\begin{equation*}
U_{0}(x)=\lim_{R\rightarrow\infty}R^{-2}u_{0}(Rx)
=\lim_{R\rightarrow\infty}R^{-2}l^{-2}u_{0}(Rlx)=l^{-2}U_{0}(lx).
\end{equation*}
This implies  that $U_{0}(x)$  satisfies condition A. Thus by Theorem
\ref{t1.2}, we  conclude  that $U(x,1)$ is a self-expending solution.

Define
\begin{equation*}
u_{R}(x,t):=R^{-2}u(Rx,R^{2}t).
\end{equation*}
It is clear that $u_{R}(x,t)$ is a solution to (\ref{e1.1}) with initial data $u_{R}(x,0)=R^{-2}u_{0}(Rx)$ satisfying Condition B.
 For any sequence $R_{i}\rightarrow +\infty$,
 consider the limitation of  $u_{R_{i}}(x,t)$.
 For $t>0$, there holds
 \begin{equation*}
 D^{2}u_{R_{i}}(x,t)=D^{2}u(R_{i}x, R^{2}_{i}t),
 \end{equation*}
 Using Theorem\ref{t1.1}, we have
 \begin{equation*}
 \lambda I\leq D^{2}u_{R_{i}}(x,t)\leq \Lambda I
 \end{equation*}
 for all $x$ and $t>0$. Moreover, according to  (\ref{e1.4}) in Theorem \ref{t1.1}, there holds
 \begin{equation*}
 \parallel D^{l}u_{R_{i}}(\cdot,t)\parallel^{2}_{C(\mathbb{R}^{n})}\leq
 C, \qquad \forall t\geq \epsilon_{0},\,\,\,l=\{3,4,5\cdots\}.
 \end{equation*}
  For any $m\geq 1, l\geq0,$ using Schauder estimates  there exists constant $C$ such that
 \begin{equation*}
 \parallel\frac{\partial^{m}}{\partial t^{m}}D^{l}u_{R_{i}}\parallel^{2}_{C(\mathbb{R}^{n})}\leq C,
 \qquad \forall t\geq \epsilon_{0},\,\,\,l=\{3,4,5\cdots\}.
 \end{equation*}
  We observe that
 \begin{equation*}
 u_{R_{i}}=R^{-2}_{i}u_{0}
 \end{equation*}
 and
 \begin{equation*}
 Du_{R_{i}}(0,0)=R^{-1}_{i}Du_{0}(0).
 \end{equation*}
 are both bounded, thus $u_{R_{i}}(0,t)$ and $Du_{R_{i}}(0,t)$ are
 uniformly bounded to $i$ for any fixed $t$. By Arzel\`{a}-Ascoli
 theorem, there exists a subsequence $\{R_{k_{i}}\}$ such that
 $u_{R_{k_{i}}}(x,t)$  converges  uniformly to a solution $\hat{U}(x,t)$ of
 (\ref{e1.6}) in compact subsets of $\mathbb{R}^{n}\times (0,
 \infty)$ and $\hat{U}(x,t)$ satisfies the estimates in Theorem
 \ref{t1.1}. Since $\dfrac{\partial \hat{U}}{\partial t}$ is uniformly
 bounded for any $t>0$, $\hat{U}(x,t)$ converges to some function $\hat{U}_{0}(x)$
 when $t\rightarrow 0$.  One can verify that
 \begin{equation*}\aligned
 \hat{U}_{0}(x)&=\lim_{t\rightarrow 0}\hat{U}(x,t)\\
 &=\lim_{t\rightarrow 0}\lim_{i\rightarrow+\infty}
 R_{i}^{-2}u(R_{i}x,R_{i}^{2}t)\\
 &=\lim_{i\rightarrow+\infty}\lim_{t\rightarrow 0}R_{i}^{-2}u(R_{i}x,R_{i}^{2}t)\\
 &=\lim_{i\rightarrow+\infty}R_{i}^{-2}u_{0}(R_{i}x) \\
 &=U_{0}(x).
 \endaligned
 \end{equation*}
 By the uniqueness result,  the above limit is
 independent of the choice of the subsequence $\{R_{i}\}$ and
 $$\hat{U}(x,t)=U(x,t).$$
 So, letting $R=\sqrt{t}$, we
 have $t^{-1}u(\sqrt{t}x,t)=u_{\sqrt{t}}(x,1)$ converging to $U(x,1)$
 uniformly in compact subsets of $\mathbb{R}^{n}$  when
 $t\rightarrow\infty$. Theorem \ref{t1.2} is established. \qed

At the end of this section, we present  the following Bernstein
theorem for the equation (\ref{e1.5}).
\begin{proposition}
Let
$$w=u-\frac{1}{2}<x,Du>.$$
If $u$ is a $C^{2}$ strictly convex solution of (\ref{e1.5})
 and $w$ can take its maximum or minimum at some point
$x\in \mathbb{R}^{n}$ with $|x|<+\infty.$  Then $u$ must be a
quadratic polynomial.
\end{proposition}
\begin{proof}
It follows from  Caffarelli's regularity theory of Monge-Amp\`{e}re
type equation   and interior Schauder estimates that $u$ is a smooth
strictly convex solution. From (\ref{e1.5})  $w$
satisfies
$$u^{ij}w_{ij}=\frac{1}{2}<x,Dw>.$$
Since $w$ can takes it's maximum or minimum at some point $x\in
\mathbb{R}^{n}$ with $|x|<+\infty,$ for every $R>0$,  by strong
maximum principle (cf. \cite{MH}), we deduce that $w$ must be
constants in $B_{R}(x)$. Then $w$ must be constants in
$\mathbb{R}^{n}$. Using  Pogorelov's Theorem \cite{P}, we show that
$u$ must be a quadratic polynomial.
\end{proof}

\section{longtime existence and convergence}

As in \cite{ACH}, we also can show that a bound on the height of the graphs is preserved along (\ref{e1.1}).

\begin{lemma}\label{l4.1}
If $u(x,t)$ is a smooth solution to (\ref{e1.1}). Then
\begin{equation}\label{e4.1}
\sup_{x\in \mathbb{R}^{n}}|Du(x,t)|^{2}\leq \sup_{x\in \mathbb{R}^{n}}|Du_{0}(x)|^{2}
\end{equation}
\end{lemma}
\begin{proof}
By (\ref{e1.1}) we have
$$\frac{\partial}{\partial t}|Du(x,t)|^{2}-\frac{1}{n}u^{ij}(|Du(x,t)|^{2})_{ij}=-\frac{2}{n}u^{pq}u_{pi}u_{qi}\leq 0.$$
Using Lemma 4.2 in \cite{X3} and  Schauder estimates  we obtain the desired results.
\end{proof}

To obtain the convergence of the flow (\ref{e1.2}), we introduce the following
 decay estimates of the high order derivatives  according to
 the solution of (\ref{e1.1}) based on Theorem \ref{t1.1} (cf. \cite{HW}).

\begin{proposition}\label{t1.2a}
Assume that   $u(x,t)$ is a strictly
convex  solution of (\ref{e1.1}) which satisfies (\ref{e1.3}) and $u(\cdot,t)$ satisfies Condition A.
Then there exist constant
$C$ only depending on $n, \lambda, \Lambda, \frac{1}{\epsilon_{0}}$ such that
\begin{equation}\label{e1.6a}
| D^{3}u(\cdot,t)|^{2}_{C(\mathbb{R}^{n})}\leq
\frac{C}{t}, \qquad \forall t\geq \epsilon_{0}.
\end{equation}
More generally, for all $l=\{3,4,5\cdots\}$ there holds
\begin{equation}\label{e1.7a}
\parallel D^{l}u(\cdot,t)\parallel^{2}_{C(\mathbb{R}^{n})}\leq
\frac{C}{t^{l-2}}, \qquad \forall t\geq\epsilon_{0}.
\end{equation}
\end{proposition}

{\bf Proof of Theorem \ref{t1.4}:}

By Theorem \ref{t1.1} and Proposition \ref{p2.1}, (\ref{e1.2}) has a longtime smooth solution.

Using (\ref{e1.7a}) and (\ref{e4.1}),  a diagonal sequence argument shows that as $t\rightarrow \infty$, $Du(x,t)$ converges subsequentially
and uniformly on compact subsets of $\mathbb{R}^{n}$ to a smooth function $Du_{\infty}$ with
\begin{equation*}
\forall y\in \mathbb{R}^{n},\quad
|D^{l}_{y}u_{\infty}|=0
\end{equation*}
for $l\geq 3$.  So $Du_{\infty}$ must be an affine linear function. Hence $(x, Du_{\infty}(x))$ has to be affine linear subspace.
It shows that the graph of the mean curvature flow (\ref{e1.2}) converges to a plane in $\mathbb{R}^{2n}_{n}$.

As the proof of Theorem 1.1 in \cite{ACH},
if  $|Du_{0}(x)|\rightarrow
0$ as $|x|\rightarrow \infty $, then the graph $(x, Du(x,t))$
converges smoothly on compact sets to the coordinate plane $(x,0)$
in $\mathbb{R}^{2n}_{n}$.
\qed

 {\bf Acknowledgment.}
 The author wishes to express his sincere gratitude to Professor  Y.L. Xin for his  valuable suggestions and comments.

\end{document}